\documentclass[11pt,english]{amsart}
\usepackage[T1]{fontenc}
\usepackage[latin9]{inputenc}
\usepackage{babel}
\usepackage{amsthm}
\usepackage{amssymb}
\usepackage{setspace}
\usepackage[unicode=true,
 bookmarks=true,bookmarksnumbered=false,bookmarksopen=false,
 breaklinks=false,pdfborder={0 0 1},backref=false,colorlinks=false]
 {hyperref}
\hypersetup{pdftitle={Isometric tuples are hyperreflexive},
 pdfauthor={Matthew Kennedy and Adam Fuller}}

\makeatletter
\theoremstyle{plain}
\newtheorem{thm}{\protect\theoremname}[section]
  \theoremstyle{plain}
  \newtheorem{prop}[thm]{\protect\propositionname}
  \theoremstyle{definition}
  \newtheorem{defn}[thm]{\protect\definitionname}
  \theoremstyle{plain}
  \newtheorem{cor}[thm]{\protect\corollaryname}
  \theoremstyle{remark}
  \newtheorem*{rem*}{\protect\remarkname}
  \theoremstyle{remark}
  \newtheorem*{acknowledgement*}{\protect\acknowledgementname}

\usepackage[T1]{fontenc}
\usepackage{ae,aecompl}

\usepackage{microtype}

\makeatother

  \providecommand{\acknowledgementname}{Acknowledgement}
  \providecommand{\corollaryname}{Corollary}
  \providecommand{\definitionname}{Definition}
  \providecommand{\propositionname}{Proposition}
  \providecommand{\remarkname}{Remark}
\providecommand{\theoremname}{Theorem}

\begin{document}

\title[Isometric tuples are hyperreflexive]{Isometric tuples are hyperreflexive}

\author{Adam H. Fuller}

\address{Department of Pure Mathematics, University of Waterloo, 200 University
Avenue West, Waterloo, Ontario N2L 3G1, Canada}

\email{a2fuller@uwaterloo.ca}

\author{Matthew Kennedy}

\address{School of Mathematics and Statistics, Carleton University, 1125 Colonel
By Drive, Ottawa, Ontario K1S 5B6, Canada}

\email{mkennedy@math.carleton.ca}
\begin{abstract}
An $n$-tuple of operators $(V_{1},\ldots,V_{n})$ acting on a Hilbert
space $\mathcal{H}$ is said to be isometric if
the row operator $(V_{1},\ldots,V_{n}):\mathcal{H}^{n}\to\mathcal{H}$
is an isometry. We prove that every isometric $n$-tuple is hyperreflexive,
in the sense of Arveson. For $n=1$, the hyperreflexivity constant
is at most $95$. For $n\geq2$, the hyperreflexivity constant is
at most $6$.
\end{abstract}



\thanks{Second author partially supported by an NSERC grant}

\maketitle

\section{Introduction and preliminaries}

An $n$-tuple of operators $(V_{1},\ldots,V_{n})$ acting on a Hilbert
space $\mathcal{H}$ is said to be isometric if the row operator $(V_{1},\ldots,V_{n}):\mathcal{H}^{n}\to\mathcal{H}$
is an isometry. This is equivalent to requiring that the operators
in the tuple are isometries with pairwise orthogonal ranges, and hence
that they satisfy the algebraic relations
\[
V_{i}^{*}V_{j}=\delta_{ij}I,\quad1\leq i,j\leq n.
\]
An isometric tuple is a natural higher-dimensional generalization
of an isometry, and these objects appear throughout mathematics and
mathematical physics (see for example \cite{Dav01,Ken11,Ken12} and
the references therein).

The notion of reflexivity, introduced by Halmos in \cite{H70} and
\cite{H71}, plays an important role in operator theory. A subspace
$\mathcal{S}$ of $\mathrm{B}\left(\mathcal{H}\right)$ is \emph{reflexive}
if 
\begin{equation}
\mathcal{S}=\mathrm{Alg(Lat}(\mathcal{S}))=\{T\in\mathrm{B}(\mathcal{H})\mid P^{\perp}TP=0\ \mbox{for every}\ P\in\mathrm{Lat}(\mathcal{S})\},\label{eq:ref-def}
\end{equation}
where $\mathrm{Lat}(\mathcal{S})$ denotes the lattice of closed invariant
subspaces for $\mathcal{S}$, and where we have identified each subspace
in $\mathrm{Lat}(\mathcal{S})$ with the corresponding projection
in $\mathrm{B}(\mathcal{H})$. The basic idea is that a reflexive
space of operators is completely determined by its invariant subspaces.

The notion of hyperreflexivity was introduced by Arveson in \cite{Arv75}
as a quantitative strengthened form of reflexivity. Before giving
the formal definition of hyperreflexivity, it will be convenient to
give another, slightly different, characterization of reflexivity.
For a subspace $\mathcal{S}$ of $\mathrm{B}(\mathcal{H})$, let 
\[
\beta(T,\mathcal{S})=\sup\{\|P^{\perp}TP\|\mid P\in\mathrm{Lat}(\mathcal{S})\},\quad T\in\mathrm{B}(\mathcal{H}).
\]
The quantity $\beta(T,\mathcal{S})$ measures the ``fit'' between
an operator $T$ in $\mathrm{B}(\mathcal{H})$ and the invariant subspace
lattice of $\mathcal{S}$. By (\ref{eq:ref-def}), we see that $\mathcal{S}$
is reflexive precisely when
\begin{equation}
\mathcal{S}=\{T\in\mathrm{B}(\mathcal{H})\mid\beta(T,\mathcal{S})=0\}.\label{eq:ref-eq}
\end{equation}
Note that this is equivalent to the existence of, for every $T$ in
$\mathrm{B}(\mathcal{H})$, a constant $C_{T}>0$ (depending on $T$)
such that $\mathrm{dist}(T,\mathcal{S})\leq C_{T}\,\beta(T,\mathcal{S})$,
where
\[
\mathrm{dist}(T,\mathcal{S})=\inf\{\|T-S\|\mid S\in\mathcal{S}\},\quad T\in\mathrm{B}(\mathcal{H}).
\]
The subspace $\mathcal{S}$ is hyperreflexive if a constant can be
chosen independent of $T$. Specifically, a subspace $\mathcal{S}$
of $\mathrm{B}(\mathcal{H})$ is \emph{hyperreflexive} if there is
a constant $C>0$ such that
\begin{equation}
\mathrm{dist}(T,\mathcal{S})\leq C\,\beta(T,\mathcal{S}),\quad T\in\mathrm{B}(\mathcal{H}).\label{eq:hyperref-ineq}
\end{equation}
In this case we say that $\mathcal{S}$ is hyperreflexive with distance
constant at most $C$. In addition, we follow the standard convention
and say that a family of operators is hyperreflexive if the weakly
closed algebra generated by the family is hyperreflexive.

Note that the inequality 
\begin{equation}
\beta(T,S)\leq\mathrm{dist}(T,\mathcal{S}),\quad T\in\mathrm{B}(\mathcal{H}),\label{eq:reverse-hyperref-ineq}
\end{equation}
always holds. This is a consequence of the fact that for arbitrary
$S$ in $\mathcal{S}$,
\[
\|P^{\perp}TP\|=\|P^{\perp}(T-S)P\|\leq\mathrm{dist}(T,\mathcal{S}).
\]
If $\mathcal{S}$ is hyperreflexive, then by (\ref{eq:hyperref-ineq})
and (\ref{eq:reverse-hyperref-ineq}), the function $T\to\beta(T,\mathcal{S})$
is equivalent to the distance function $T\to\mathrm{dist}(T,\mathcal{S})$
on $\mathrm{B}(\mathcal{H})$. In particular, it follows from (\ref{eq:ref-eq})
that $\mathcal{S}$ is reflexive. It was shown in \cite{KL86}, however,
that the converse is false, and hence that hyperreflexivity is a strictly
stronger property than reflexivity.

In addition to Arveson's work, the notion of hyperreflexivity has
been studied by many other authors. The work of Christensen in \cite{Chr82}
(see also \cite{Chr77}) established that the hyperreflexivity of
a von Neumann algebra is equivalent to a positive solution to the
Derivation Problem. In \cite{Dav87}, Davidson proved that the unilateral
shift is hyperreflexive, and in \cite{KP05}, Kli\'{s} and Ptak generalized
this result to a more general class of operators called quasinormal
operators.

In \cite{DP99}, Davidson and Pitts generalized this result in a different
direction, and proved the hyperreflexivity of a class of isometric
tuples that can be seen as a higher dimensional generalization of
the unilateral shift. Subsequently, in \cite{Ber98}, Bercovici established
the hyperreflexivity of much larger class of algebras, and substantially
improved the distance constant from \cite{DP99}. See also the papers
\cite{LS75}, \cite{Ros82}, \cite{KL85}, \cite{KL86}, \cite{MP05},
\cite{DL06}, \cite{KP06}, \cite{JP06} and \cite{PP12}.

Despite these results, hyperreflexivity is still not well understood.
In fact, it is still not known whether many interesting and naturally
occuring examples are hyperreflexive, or even reflexive. In \cite{DKP01},
which was a followup to \cite{DP99}, Davidson, Katsoulis and Pitts
posed the problem of whether every isometric tuple is hyperreflexive.
Recently, in \cite{Ken11} (see also \cite{Ken12}) a partial solution
to this problem was given, and it was established that the class of
absolutely continuous isometric tuples is hyperreflexive (see Section
\ref{sec:isom-tuples} below for the definition of absolute continuity).

The main result in this paper is the following theorem, which completely
resolves the problem from \cite{DKP01}, and more generally, provides
the first large class of hyperreflexive examples in multivariate operator
theory. 
\begin{thm}
Every isometric $n$-tuple is hyperreflexive. For $n=1$, the distance
constant is at most $95$. For $n\geq2$, the distance constant is
at most $6$.
\end{thm}

In addition to this introduction, this paper has two other sections.
In Section \ref{sec:isometries}, we consider the classical case of
a single isometry. In Section \ref{sec:isom-tuples}, we consider
isometric tuples and establish our main result.

\section{\label{sec:isometries}The hyperreflexivity of isometries}

In this section we consider the case of a single isometry. Since an
isometry is a special case of a quasinormal operator, the main result
from \cite{KP05} implies that every isometry is hyperreflexive. However,
we provide a simple proof of this result here, based on the main result
in \cite{Dav87}, and obtain a better distance constant.

We will require the following classical result, which follows from
the Wold decomposition and the Lebesgue decomposition of a measure.
The details can be found in, for example, \cite{NFBK10}.
\begin{thm}
[Lebesgue-von Neumann-Wold Decomposition]\label{thm:classical-lebesgue-vn-wold}Let
$V$ be an isometry. Then $V$ can be decomposed as 
\[
V=V_{u}\oplus V_{a}\oplus V_{s},
\]
where $V_{u}$ is a unilateral shift, $V_{a}$ is a unitary with an
absolutely continuous spectral measure, and $V_{s}$ is a unitary
with a singular spectral measure. Note that absolute continuity and
singularity are with respect to Lebesgue measure.
\end{thm}

We will also require two results from the literature. The first result,
proved by Kraus and Larson in \cite{KL86}, and independently by Davidson
in \cite{Dav87}, shows that in certain cases, hyperreflexivity is
inherited by subspaces. Recall that a weak-{*} closed subspace $\mathcal{S}$
of $\mathcal{B}(\mathcal{H})$ has property $\mathbb{A}_{1}(1)$ if
for every $\varepsilon>0$, and every weak-{*} continuous linear functional
$\varphi$ on $\mathcal{S}$, there are vectors $x$ and $y$ in $\mathcal{H}$
such that $\varphi(S)=\langle Sx,y\rangle$ for all $S$ in $\mathcal{S}$,
and $\|x\|\|y\|<(1+\varepsilon)\|\varphi\|$.
\begin{thm}[Kraus-Larson, Davidson]
\label{thm:hereditary-hyperreflexive} Let $\mathcal{S}$ be a hyperreflexive
subspace of $\mathrm{B}(\mathcal{H})$ with distance constant at most
$C$. Suppose furthermore that $\mathcal{S}$ has property $\mathbb{A}_{1}(1)$.
Then any weak-{*} closed subspace of $\mathcal{S}$ is hyperreflexive
with distance constant at most $2C+1$.
\end{thm}
The second result, proved by Kli\'{s} and Ptak in \cite{KP06}, shows
that hyperreflexivity is preserved under taking direct sums.
\begin{thm}[Kli\'{s}-Ptak]
\label{thm:directsum-hyperreflexive} Let $\{\mathcal{A}_{n}|n\in\mathbb{N}\}$
be a family of hyperreflexive subspaces with hyperreflexivity constants
bounded by $C$. Then the direct sum $\mathcal{A}=\bigoplus_{n\in\mathbb{N}}\mathcal{A}_{n}$
is hyperreflexive with distance constant at most $2+3C.$
\end{thm}

We are now able to prove the main result in this section.
\begin{prop}
\label{prop:n=00003D1} Every isometry is hyperreflexive with distance
constant at most $95$.\end{prop}
\begin{proof}
Let $V$ be an isometry with Lebesgue-von Neumann-Wold decomposition
$V=V_{u}\oplus V_{a}\oplus V_{s}$ as in Theorem \ref{thm:classical-lebesgue-vn-wold}.
Then it's clear that 
\[
\mathrm{W}(V)\subseteq\mathrm{W}(V_{u})\oplus\mathrm{W}^{*}(V_{a})\oplus\mathrm{W}^{*}(V_{s}).
\]

We will first show that $W(V_{u})$ is hyperreflexive with distance
constant at most $15$. By \cite{Dav87}, unilateral shifts of multiplicity
$1$ are hyperreflexive, and by \cite{KP05}, this distance constant
is at most $13$. Now suppose $V_{u}$ has multiplicity $\alpha$
and acts on the Hilbert space $\mathcal{H}=\bigoplus\mathcal{H}_{i}$.
Let 
\[
\mathcal{M}=\bigoplus\mathcal{B}(\mathcal{H}_{i})\subseteq\mathcal{B}(\mathcal{H}).
\]
Then $\mathcal{M}$ is an injective von Neumann algebra with abelian
commutant $\mathcal{M}'$. Denote by $\Phi$ the expectation from
$\mathcal{B}(\mathcal{H})$ onto $\mathcal{M}$. Take any $T\in\mathcal{B}(\mathcal{H}).$
Following the methods of \cite{Ros82}, and noting that $\beta_{\mathcal{M}}(T)\leq\beta_{W(V_{u})}(T)$,
we see that
\[
\|T-\Phi(T)\|\leq2\beta_{W(V_{u})}(T).
\]
We have now that
\begin{align*}
\mathrm{dist}(T,W(V_{u})) & \leq\mathrm{dist}(T-\Phi(T),W(V_{u}))+\mathrm{dist}(\Phi(T),W(V_{u}))\\
 & \leq\|T-\Phi(T)\|+\mathrm{dist}(\Phi(T),W(V_{u})).\\
 & \leq2\beta_{W(V_{u})}(T)+13\beta_{W(V_{u})}(T)\\
 & =15\beta_{W(V_{u})}(T).
\end{align*}

Since $\mathrm{W}^{*}(V_{a})$ and $\mathrm{W}^{*}(V_{s})$ are abelian
von Neuman algebras, they are both hyperreflexive with distance constant
at most $2$ by \cite{Ros82}. Again, since $\mathrm{W}^{*}(V_{a})$
and $\mathrm{W}^{*}(V_{s})$ are abelian von Neumann algebras, they
also have property $\mathbb{A}_{1}(1)$ (see for example Theorem 60.1
of \cite{Con00}). Therefore, by\ref{thm:hereditary-hyperreflexive},
$ $$\mathrm{W}(V_{a})$ and $\mathrm{W}(V_{s})$ are hyperreflexive
with distance constant at most $5.$

By Theorem \ref{thm:directsum-hyperreflexive}, the algebra $\mathrm{W}(V_{u})\oplus\mathrm{W}^{*}(V_{a})\oplus\mathrm{W}^{*}(V_{s})$
is hyperreflexive with distance constant at most $47$. Finally, by
Theorem \ref{thm:hereditary-hyperreflexive} we get that $W(V)$ is
hyperreflexive with distance constant at most $95$.
\end{proof}

\section{\label{sec:isom-tuples}The hyperreflexivity of isometric tuples}

In this section, we will prove that every isometric tuple is hyperreflexive,
which is the main result of this paper. No straightforward generalization
of the approach taken in Section \ref{sec:isometries} will suffice
here, because the structure of an isometric tuple can be substantially
more complicated than the structure of an isometry. We will utilize
the structure theorem for isometric tuples from \cite{Ken12}. Before
stating that result, we need to introduce some terminology.

Let $\mathcal{F}_{n}$ denote the full Fock space over $\mathbb{C}^{n}$.
That is,
\[
\mathcal{F}_{n}=\mathbb{C}\oplus\mathbb{C}^{n}\oplus(\mathbb{C}^{n})^{\otimes2}\oplus(\mathbb{C}^{n})^{\otimes3}\oplus\ldots
\]
For a fixed orthonormal basis $\{\xi_{1},\ldots,\xi_{n}\}$ of $\mathbb{C}^{n}$,
let $L_{1},\ldots,L_{n}$ denote the left creation operators on $\mathcal{F}_{d}$.
That is,
\[
L_{i}(\xi_{j_{1}}\otimes\ldots\otimes\xi_{j_{k}})=\xi_{i}\otimes\xi_{j_{1}}\otimes\ldots\otimes\xi_{j_{k}},\quad1\leq i,j_{1},\ldots,j_{k}\leq n.
\]
The \emph{unilateral $n$-shift} is the isometric tuple $L=(L_{1},\ldots,L_{n})$,
and the \emph{noncommutative analytic Toeplitz algebra} $\mathcal{L}_{n}$
is the weakly closed algebra $\mathrm{W}(L)$ generated by $L_{1},\ldots,L_{n}$.

The motivation for these names is the fact that, for $n=1$, $L$
can be identified with the classical unilateral shift, and $\mathcal{L}_{n}$
can be identified with the classical algebra $H^{\infty}$ of bounded
analytic functions on the complex unit disk. The study of these objects
was initiated by Popescu in \cite{Pop91} and \cite{Pop96}, where
they were shown to possess a great deal of analytic structure. They
were also studied in detail by Davidson and Pitts in \cite{DP98}
and \cite{DP99}, and more recently in \cite{Ken11} and \cite{Ken12}.
It has become clear in recent years that these objects play a central
role in multivariate operator theory (see for example \cite{Dav01}).

\begin{defn}
Let $V=\left(V_{1},\ldots,V_{n}\right)$ be an isometric $n$-tuple,
for $n\geq2$. Then 
\begin{enumerate}
\item $V$ is a \emph{unilateral shift} if it is unitarily equivalent to
an ampliation of the unilateral $n$-shift,
\item $V$ is \emph{absolutely continuous} if the weakly closed algebra
$\mathrm{W}(V_{1},\ldots,V_{n})$ is isomorphic to the noncommutative
analytic Toeplitz algebra $\mathcal{L}_{n}$,
\item $V$ is \emph{singular} if the weakly closed algebra $\mathrm{W}(V_{1},\ldots,V_{n})$
is a von Neumann algebra,
\item $V$ is of \emph{dilation type} if it has no summand that is absolutely
continuous or singular.
\end{enumerate}
\end{defn}

This definition merits a few remarks. First, the existence of singular
isometric tuples was an open problem for some time before it was established
by Read in \cite{Rea05} (see also \cite{Dav06} for an exposition).
Second, isometric tuples of dilation type do not appear in the classical
setting of a single isometry. In the higher dimensional setting, they
arise as the minimal isometric dilation of contractive tuples (see
\cite{Ken12} for more details).

The following theorem is the Lebesgue-von Neumann-Wold decomposition
of an isometric tuple from \cite{Ken12}.

\begin{thm}[Lebesgue-von Neumann-Wold Decomposition]
\label{thm:decomp-isom-tuple}Let $V=\left(V_{1},\ldots,V_{n}\right)$
be an isometric $n$-tuple. For $n\geq2$, $V$ can be decomposed
as 
\begin{equation}
V=V_{u}\oplus V_{a}\oplus V_{s}\oplus V_{d},\label{eq:decomp-isom-tuple}
\end{equation}
where $V_{u}$ is a unilateral $n$-shift, $V_{a}$ is an absolutely
continuous $n$-tuple, $V_{s}$ is a singular $n$-tuple and $V_{d}$
is an isometric $n$-tuple of dilation type.
\end{thm}

It follows from the results in \cite{Ken11} and \cite{Ken12} that
an isometric tuple $V$ is absolutely continuous if and only if $V_{s}=0$
and $V_{d}=0$ in the decomposition (\ref{eq:decomp-isom-tuple}).
The following result was also obtained in \cite{Ken12}.
\begin{prop}
\label{prop:analytic-fsa-hyperreflexive}For $n\geq2$, every absolutely
continuous $n$-tuple is hyperreflexive with distance constant at
most $3$. 
\end{prop}

To handle the case of singular isometric tuples and isometric tuples
of dilation type, we will need a better understanding of the von Neumann
algebra generated by certain isometric tuples. An isometric tuple
$V=(V_{1},\ldots,V_{n})$ is said to be \emph{unitary} if the row
operator $(V_{1},\ldots,V_{n}):\mathcal{H}^{n}\to\mathcal{H}$ is
surjective. Since, in the classical case, a unitary is precisely a
surjective isometry, this is a natural higher dimensional generalization
of the notion of a unitary.
\begin{prop}
\label{prop:commutant-vn-alg-injective}For $n\geq2$, the commutant
of the von Neumann algebra generated by a unitary $n$-tuple is injective.\end{prop}
\begin{proof}
Let $U=(U_{1},\ldots,U_{n})$ be a unitary $n$-tuple acting on a
Hilbert space $\mathcal{H}$, and let $\mathcal{M}=\mathrm{W}^{*}(U_{1},\ldots,U_{n})$
denote the von Neumann algebra generated by $U$. Define a completely
positive map $\Gamma:\mathrm{B}\left(\mathcal{H}\right)\to\mathrm{B}\left(\mathcal{H}\right)$
by 
\[
\Gamma\left(T\right)=\sum_{i=1}^{n}U_{i}TU_{i}^{*},\quad T\in\mathrm{B}\left(\mathcal{H}\right),
\]
and let 
\[
\mathcal{F}\left(\Gamma\right)=\left\{ T\in\mathrm{B}\left(\mathcal{H}\right)\mid\Gamma\left(T\right)=T\right\} 
\]
denote the set of fixed points of $\Gamma$. Let $\mathcal{M}'$ denote
the commutant of $\mathcal{M}$. We will first show that $\mathcal{M}'=\mathcal{F}\left(\Gamma\right)$.

Suppose first that $T$ belongs to $\mathcal{M}'$. Then 
\[
\Gamma\left(T\right)=\sum_{i=1}^{n}U_{i}TU_{i}^{*}=T\sum_{i=1}^{n}U_{i}U_{i}^{*}=T,
\]
and hence $T\in\mathcal{F}\left(\Gamma\right)$. Now suppose that
$T$ belongs to $\mathcal{F}\left(\Gamma\right)$. Then 
\[
T=\Gamma\left(T\right)=\sum_{i=1}^{n}U_{i}TU_{i}^{*}.
\]
Therefore, for each $1\leq j\leq n$, multiplying $T$ on the left
by $U_{j}^{*}$ gives $U_{j}^{*}T=TU_{j}^{*}$, and multiplying $T$
on the right by $U_{j}$ gives $TU_{j}=U_{j}T$. Hence $T\in\mathcal{M}'$.
Thus we see that $\mathcal{M}'=\mathcal{F}\left(\Gamma\right)$.

By Lemma 2 of \cite{Arv72}, there is a completely positive and contractive
idempotent map $\Phi:\mathrm{B}\left(\mathcal{H}\right)\to\mathrm{B}\left(\mathcal{H}\right)$
with the property that $\mathrm{Ran}\,\Phi=\mathcal{F}\left(\Gamma\right)$.
We note that the existence of the map $\Phi$ has become a standard
tool in the theory of completely positive maps. It can be realized
as the strong operator limit 
\[
\Phi\left(X\right)=\lim_{k\to\infty}\Gamma^{k}\left(X\right),\quad X\in\mathrm{B}\left(\mathcal{H}\right).
\]
Note that since $U$ is unitary, $\Gamma\left(I\right)=I$. Therefore,
$\Phi\left(I\right)=I$ which gives $\left\Vert \Phi\right\Vert =1$.
It follows that $\Phi$ is a projection from $\mathrm{B}\left(\mathcal{H}\right)$
to $\mathcal{M}'$, and hence that $\mathcal{M}'$ is injective. 
\end{proof}

We note that, in addition to the direct proof given here, the previous
result can also be obtained using some deep results from the theory
of C{*}-algebras and von Neumann algebras. Indeed, the C{*}-algebra
generated by a unitary $n$-tuple is isomorphic to the Cuntz algebra
$\mathcal{O}_{n}$, and hence is nuclear \cite{Cu77}. Therefore,
the results in \cite{CE78} imply that the von Neumann algebra generated
by a unitary $n$-tuple is injective, and it follows from the theory
of injective von Neumann algebras that the commutant is injective.

\begin{cor}
\label{cor:singular-fsa-hyperreflexive} For $n\geq2$, every singular
isometric $n$-tuple is hyperreflexive with distance constant at most
$4$.\end{cor}
\begin{proof}
By definition, the weakly closed algebra generated by a singular isometric
tuple is a von Neumann algebra, and by Proposition \ref{prop:commutant-vn-alg-injective},
the commutant of this von Neumann algebra is injective. By \cite{Chr77},
this algebra is hyperreflexive with distance constant at most $4$. 
\end{proof}

We now have everything we need to prove the main result.
\begin{thm}
\label{thm:main-theorem} Every isometric $n$-tuple is hyperreflexive.
For $n=1$, the distance constant is at most $95$. For $n\geq2$,
the distance constant is at most $6$.\end{thm}
\begin{proof}
For $n=1$, the result follows from Proposition \ref{prop:n=00003D1}.
Therefore we let $V=(V_{1},\ldots,V_{n})$ be an isometric $n$-tuple
of dilation type acting on a Hilbert space $\mathcal{H}$, for $n\geq2$.
Let $\mathcal{M}$ denote the von Neumann algebra $\mathrm{W}^{*}(V)$
generated by $V$, and let $\mathcal{S}$ be the weakly closed algebra
$\mathrm{W}(V)$, i.e. the free semigroup algebra, generated by $V$
. By the structure theorem from \cite{DKP01}, there is a projection
$P$ in $\mathcal{S}$, with range coinvariant for $\mathcal{S}$,
such that $\mathcal{S}P=\mathcal{M}P$.

If $P=0$, then $V$ is analytic, and the result follows from Proposition
\ref{prop:analytic-fsa-hyperreflexive}. On the other hand, if $P=I$,
then $V$ is singular, and the result follows from Corollary \ref{cor:singular-fsa-hyperreflexive}.
Hence we can suppose that $P\ne0$ and $P\ne I$. Note that $\mathcal{S}P$
and $\mathcal{S}P^{\perp}$ are both weakly closed algebras, and that
we can write $\mathcal{S}=\mathcal{S}P+\mathcal{S}P^{\perp}$. Note
also that if $S$ and $R$ belong to $\mathcal{S}$, then $SP$ and
$RP^{\perp}$ also belong to $\mathcal{S}$. Hence for $T$ in $\mathrm{B}(\mathcal{H})$,
\begin{align*}
\mathrm{dist}(T,\mathcal{S})^{2} & =\inf\{\|T-S\|^{2}\mid S\in\mathcal{S}\}\\
 & =\inf\{\|TP-SP+TP^{\perp}-SP^{\perp}\|^{2}\mid S\in\mathcal{S}\}\\
 & =\inf\{\|TP-SP+TP^{\perp}-RP^{\perp}\|^{2}\mid S,R\in\mathcal{S}\}\\
 & =\inf\{\|PT^{*}-PS^{*}+P^{\perp}T^{*}-P^{\perp}R^{*}\|^{2}\mid S,R\in\mathcal{S}\}\\
 & =\inf\{\|TP-SP\|^{2}+\|TP^{\perp}-RP^{\perp}\|^{2}\mid S,R\in\mathcal{S}\}\\
 & =\inf\{\|TP-SP\|^{2}\mid S\in\mathcal{S}\}+\inf\{\|TP^{\perp}-RP^{\perp}\|^{2}\mid R\in\mathcal{S}\}\\
 & =\mathrm{dist}(TP,\mathcal{S}P)^{2}+\mathrm{dist}(TP^{\perp},\mathcal{S}P^{\perp})^{2}.
\end{align*}
In order to show that $\mathcal{S}$ is hyperreflexive, we will bound
 $\mathrm{dist}(TP,\mathcal{S}P)$ and $\mathrm{dist}(TP^{\perp},\mathcal{S}P^{\perp})$
separately.

First, we consider the value of $\mathrm{dist}(TP,\mathcal{S}P)$.
By the argument from the proof of Corollary \ref{cor:singular-fsa-hyperreflexive},
the von Neumann algebra $\mathcal{M}$ is hyperreflexive with distance
constant at most $4$. Therefore, 
\begin{equation}
\mathrm{dist}(TP,\mathcal{S}P)=\mathrm{dist}(TP,\mathcal{M}P)=\mathrm{dist}(TP,\mathcal{M})\leq4\beta(TP,\mathcal{M}).\label{eq:main-theorem-0}
\end{equation}
For $x$ and $y$ in $\mathcal{H}$, let $xy^{*}$ denote the linear
functional on $\mathrm{B}(\mathcal{H})$ defined by
\[
(xy^{*})(T)=(Tx,y),\quad T\in\mathrm{B}(\mathcal{H}).
\]
Then we have the inequality 
\begin{align}
\beta(TP,\mathcal{M}) & =\sup\{|\langle TPx,y\rangle|\mid xy^{*}\in\mathcal{M}_{\perp},\ \|x\|\|y\|\leq1\}\label{eq:main-theorem-1}\\
 & =\sup\{|\langle TP(Px),y\rangle|\mid(Px)y^{*}\in\mathcal{M}_{\perp},\ \|x\|\|y\|\leq1\}\nonumber \\
 & \leq\sup\{|\langle TPx,y\rangle|\mid xy^{*}\in(\mathcal{M}P)_{\perp},\ \|x\|\|y\|\leq1\}\nonumber \\
 & =\beta(TP,\mathcal{M}P),\nonumber 
\end{align}
and similarly, 
\begin{align}
\beta(TP,\mathcal{S}P) & =\sup\{|\langle TPx,y\rangle|\mid xy^{*}\in(\mathcal{S}P)_{\perp},\ \|x\|\|y\|\leq1\}\label{eq:main-theorem-2}\\
 & =\sup\{|\langle T(Px),y\rangle|\mid(Px)y^{*}\in\mathcal{S}_{\perp},\ \|x\|\|y\|\leq1\}\nonumber \\
 & \leq\sup\{|\langle Tx,y\rangle|\mid xy^{*}\in\mathcal{S}_{\perp},\ \|x\|\|y\|\leq1\}\nonumber \\
 & =\beta(T,\mathcal{S}).\nonumber 
\end{align}
Putting (\ref{eq:main-theorem-0}), (\ref{eq:main-theorem-1}) and
(\ref{eq:main-theorem-2}) together gives
\begin{equation}
\mathrm{dist}(TP,\mathcal{S}P)\leq4\beta(TP,\mathcal{M})\leq4\beta(TP,\mathcal{M}P)\leq4\beta(T,\mathcal{S}).\label{eq:main-theorem-a}
\end{equation}

Now we consider the value of $\mathrm{dist}(TP^{\perp},\mathcal{S}P^{\perp})$.
To obtain a bound, we could appeal to Theorem \ref{thm:directsum-hyperreflexive},
but by calculating directly, we can obtain a better distance constant.
By Proposition 6.2 of \cite{Ken12}, the compression $P^{\perp}\mathcal{S}\mid_{\mathrm{ran}(P^{\perp})}$
is absolutely continuous. Hence by Proposition \ref{prop:analytic-fsa-hyperreflexive},
the algebra $\mathcal{S}P^{\perp}$ is also hyperreflexive with distance
constant at most $3$. Since the range of $P^{\perp}$ is invariant
for $\mathcal{S}$, $SP^{\perp}=P^{\perp}SP^{\perp}$ for every $S$
in $\mathcal{S}$. Hence
\begin{align}
\mathrm{dist}(TP^{\perp},\mathcal{S}P^{\perp}) & =\inf\{\|TP^{\perp}-SP^{\perp}\|\mid S\in\mathcal{S}\}\label{eq:main-theorem-3}\\
 & \leq\inf\{\|PTP^{\perp}\|+\|P^{\perp}TP^{\perp}-SP^{\perp}\|\mid S\in\mathcal{S}\}\nonumber \\
 & =\|PTP^{\perp}\|+\inf\{\|P^{\perp}TP^{\perp}-SP^{\perp}\|\mid S\in\mathcal{S}\}\nonumber \\
 & =\|PTP^{\perp}\|+\mathrm{dist}(P^{\perp}TP^{\perp},\mathcal{S}P^{\perp}).\nonumber 
\end{align}
But since $P^{\perp}\in\mathrm{Lat}(\mathcal{S})$, $\|PTP^{\perp}\|\leq\beta(T,\mathcal{S})$,
and since $\mathcal{S}P^{\perp}$is hyperreflexive with distance constant
at most $3$, 
\begin{equation}
\mathrm{dist}(P^{\perp}TP^{\perp},\mathcal{S}P^{\perp})\leq3\beta(P^{\perp}TP^{\perp},\mathcal{S}P^{\perp}).\label{eq:main-theorem-4}
\end{equation}
By an argument similar to (\ref{eq:main-theorem-1}) and (\ref{eq:main-theorem-2}),
$\beta(P^{\perp}TP^{\perp},\mathcal{S}P^{\perp})\leq\beta(T,\mathcal{S})$.
Hence putting (\ref{eq:main-theorem-3}) and (\ref{eq:main-theorem-4})
together gives 
\begin{equation}
\mathrm{dist}(TP^{\perp},\mathcal{S}P^{\perp})\leq\beta(T,\mathcal{S})+3\beta(P^{\perp}TP^{\perp},\mathcal{S}P^{\perp})\leq4\beta(T,\mathcal{S}).\label{eq:main-theorem-b}
\end{equation}
Combining (\ref{eq:main-theorem-a}) and (\ref{eq:main-theorem-b}),
we see that $\mathcal{S}$ is hyperreflexive with distance constant
at most $4\sqrt{2}<6$.\end{proof}
\begin{rem*}
We note that by replacing (\ref{eq:main-theorem-3}) in the proof
of Theorem \ref{thm:main-theorem} with the inequality 
\[
\mathrm{dist}(TP^{\perp},\mathcal{S}P^{\perp})\leq(\mathrm{dist}(PTP^{\perp},P\mathcal{S}P^{\perp})^{2}+\mathrm{dist}(P^{\perp}TP^{\perp},P^{\perp}\mathcal{S}P^{\perp})^{2})^{1/2},
\]
we could obtain the slightly better distance constant $\sqrt{26}\approx5.1$.
\end{rem*}

\begin{acknowledgement*}
The authors are grateful to Ken Davidson for many helpful comments
and suggestions which, in particular, helped us improve the distance
constants obtained in Theorem \ref{thm:main-theorem}. The authors
are also grateful to Dilian Yang for correcting an early draft of
this paper.
\end{acknowledgement*}

\end{document}